\documentclass{article}

\usepackage[utf8]{inputenc}
\usepackage{geometry,amsmath,amssymb,graphicx,float,tikz,mathabx,mathrsfs, amsthm, ytableau, extarrows}
\usetikzlibrary{fit, backgrounds}
\usepackage{hyperref}
\hypersetup{
    colorlinks=true,
    linkcolor=black,
    filecolor=magenta,      
    urlcolor=black,
    pdftitle={Overleaf Example},
    pdfpagemode=FullScreen,
    citecolor=blue
    }
\usepackage{verbatim}
\usepackage{graphicx}
\usepackage{bbm}
\usepackage{pythonhighlight}
\usepackage{authblk}
\usetikzlibrary{arrows.meta,calc,positioning}

\usepackage{float}
\usepackage{algpseudocode}
\usepackage{algorithm}

\newtheorem{theorem}{Theorem}[section]

\newtheorem{lemma}{Lemma}[section]
\newtheorem{corollary}{Corollary}[section]
\theoremstyle{definition}

\DeclareMathOperator{\Deg}{deg}
\newcommand{\ming}{%
  \min \{ \Deg_G(u), \Deg_G(v) \}%
}

\newcommand{\mingen}[3]{%
  \min \{ \Deg_{#1}(#2), \Deg_{#1}(#3) \}%
}

\newcommand{\degr}[2]{%
  \Deg_{#1}(#2)%
}

\usepackage{enumitem}
\newlist{theoremenum}{enumerate}{1}
\setlist[theoremenum,1]{label=(\thetheorem.\arabic*), ref=\thetheorem.\arabic*}

\newcounter{cases}
\newcounter{subcases}[cases]

\title{Ideally Connected Cographs and Chordal Graphs}
\author[1]{Richter Jordaan\thanks{Georgia Institute of Technology, Atlanta, GA 30332 \\ 
Email: \href{http://www.overleaf.com}{{rjordaan3@gatech.edu}}}}
\date{\vspace{-7ex}}

\begin{document}
\maketitle

 \begin{abstract} For distinct vertices $u,v$ in a graph $G$, let $\kappa_G(u,v)$ denote the maximum number of internally disjoint $u$-$v$ paths in $G$. Then, $\kappa_G(u,v) \leq \min\{ \mbox{deg}_G(u), \mbox{deg}_G(v) \}$. If equality is attained for every pair of vertices in $G$, then $G$ is called \textit{ideally connected}. In this paper, we characterize the ideally connected graphs in two well-known graph classes: the cographs and the chordal graphs. We show that the ideally connected cographs are precisely the $2K_2$-free cographs, and the ideally connected chordal graphs are precisely the threshold graphs, the graphs that can be constructed from the single-vertex graph by repeatedly adding either an isolated vertex or a dominating vertex.
 \end{abstract}

\section{Introduction}
\subsection{Background}
The most well-known measure of the connectedness of a graph $G$  is its vertex connectivity $\kappa(G)$, which is
defined as the minimum size of a vertex subset whose removal disconnects $G$ (unless $G$ is isomorphic to a complete graph, in which case $\kappa(G)=|V(G)|-1)$. Menger's theorem shows that $\kappa(G)$ is equal to the maximum number of internally disjoint paths that can always be found between any pair of vertices in $G$. For distinct vertices $u,v \in V(G)$, we let $\kappa_G(u,v)$ denote the maximum number of internally disjoint $u$-$v$ paths in $G$, 
so Menger's theorem can be formulated as \[ \kappa(G) = \min_{ u,v \in V(G)}{\kappa_G(u,v)} \]
Then, for distinct vertices $u,v \in V(G)$, we have $\kappa(G) \leq \kappa_G(u,v) \leq \ming$. If $\kappa_G(u,v) = \ming$ for every pair of vertices $u,v \in V(G)$, we say that $G$ is \textit{ideally connected}.

Ideal connectedness has arisen independently, and under different names, in work on various properties and quantities reflecting a graph's connectivity. Beineke, Oellermann, and Pippert \cite{Beineke2002} introduced the average connectivity parameter, which for a graph $G$ is defined as the average (as opposed to the minimum) of $\kappa_G(u,v)$ among all pairs of vertices in $G$. Whereas the standard vertex connectivity measures the ``worst case" scenario of $\kappa_G(u,v)$ among all pairs of vertices, the average connectivity can be a more global and more refined measure of a graph's overall connectedness in applications, particularly in network modeling \cite{Bagga1993}. Ideal connectedness is the ``best case" outcome for the average connectivity of a graph, given its degree sequence. It turns out that ideally connected graphs frequently arise as the extremal examples for average connectivity under certain further conditions. For instance, Beineke, Oellermann, and Pippert \cite{Beineke2002} showed that among all graphs with a fixed number of vertices and edges, the ones of largest average connectivity are ideally connected; for other examples, see 
 \cite{Casablanca2021,Dankelmann2003,Mol2021}. Hence, it is of interest to study the structure of ideally connected graphs themselves, since in some sense this amounts to studying the structure of graphs with largest average connectivity.

Hellwig and Volkmann \cite{Hellwig2006} studied ideally connected graphs under the name of ``maximally locally connected graphs" as a generalization of maximally connected graphs, where a graph is called \textit{maximally connected} if its vertex connectivity equals its minimum degree. Among other results, Hellwig and Volkmann give a minimum degree condition that is sufficient to ensure that a graph is ideally connected. Finally, in theoretical computer science, various researchers have studied ideally connected graphs under the name of ``strongly Menger-connected graphs" in the context of parallel routing in computer networks. In fact, researchers have investigated the more general problem of finding internally disjoint paths between two vertices in networks with faults. A graph $G$ is called \textit{strongly $m$-Menger connected} if for any vertex set $F \subseteq V(G)$ of size at most $m$, the subgraph of $G$ induced by the vertices in $V(G) \setminus F$ is ideally connected. Then, a graph $G$ is ideally connected if and only if it is strongly $0$-Menger connected. Past work has focused on showing strong Menger connectivity in particular graph classes, particularly for various hypercubes networks \cite{Cheng2014, Gu2019, Shih2008} and permutation star graphs \cite{OhChen}.

\subsection{Overview of Results}
In this paper we study ideal connectedness in two well-known graph classes: cographs and chordal graphs. A graph $G$ is called a \textit{cograph} if $G$ does not contain an induced path on four vertices. A graph $G$ is called \textit{chordal} if every cycle in $G$ has a chord. Cographs and chordal graphs are both well-studied graph classes that have widespread connections across graph theory and computer science, such as in the study of perfect graphs and efficient algorithms; see \cite{Golumbic1980} for further background. 

For a fixed collection $\mathcal{H}$ of finite graphs, we say that $G$ is \textit{$\mathcal{H}$-free} if there is no induced subgraph of $G$ that is isomorphic to a graph in $\mathcal{H}$. If $\mathcal{H}$ consists of a single element $H$, we say that a graph $G$ is $H$-free if it is $\mathcal{H}$-free where $\mathcal{H}=\{ H \}$. For instance, cographs are exactly the $P_4$-free graphs and chordal graphs are exactly the $\{ C_k: k \geq 4\}$-free graphs.

In Section \ref{sec: cograph} we characterize the ideally connected cographs as follows.
\begin{theorem}\label{thm: intro cograph} Let $G$ be a cograph. Then $G$ is ideally connected if and only if $G$ is $2K_2$-free.
\end{theorem}

We then study ideal connectedness in chordal graphs. One special subclass of chordal graphs is the class of \textit{split graphs}, the graphs whose vertices can be partitioned into a clique and an independent set. If the vertices $V(G)$ of a graph $G$ can be partitioned into a clique $K$ and independent set $I$, we say that $G$ is \textit{$(K,I)$-split}. If $U \subseteq V(G)$ is a set of vertices of $G$ and the vertices in $U$ can be listed as $v_1,v_2,\ldots,v_{|U|}$ so that for each $i \in \{1,2,\ldots, |U|-1\}$ we have $N_G(v_i) \subseteq N_G(v_{i+1})$, then we say that $U$ has the \textit{nested neighborhood property} in $G$. 
If $G$ is a $(K,I)$-split graph and $I$ has the nested neighborhood property in $G$, then $G$ is called a  \textit{threshold graph}. Threshold graphs are also sometimes referred to as \textit{nested split graphs}. Figure \ref{fig: threshold example} shows an example of a threshold graph on sixteen vertices.  


\begin{figure}[h]
\begin{center}
\begin{tikzpicture}

    \foreach \i in {1,...,12} {
        \node[draw=black, fill=white, circle, inner sep=1.5pt] (b\i) at ({2 + 1.5*cos(30*\i)}, {3 + 1.5*sin(30*\i)}) {};
    }

    \foreach \i in {1,...,12} {
        \foreach \j in {1,...,12} {
            \ifnum\i<\j
                \draw[black!30] (b\i) -- (b\j);
            \fi
        }
    }

    \def\xshift{6.0}

    \node[draw=black, fill=black, circle, inner sep=1.5pt] (bv1) at (\xshift,3.1) {};
    \node[draw=black, fill=black, circle, inner sep=1.5pt] (bv2) at (\xshift,3.6) {};
    \foreach \i in {12,1,2} {
        \draw[black, thick] (bv1) -- (b\i);
        \draw[black, thick] (bv2) -- (b\i);
    }

    \node[draw=black, fill=black, circle, inner sep=1.5pt] (gv) at (\xshift,2.4) {};
    \foreach \i in {11,12,1,2} {
        \draw[black, thick] (gv) -- (b\i);
    }

    \node[draw=black, fill=black, circle, inner sep=1.5pt] (rv) at (\xshift,1.7) {};
    \foreach \i in {10,11,12,1} {
        \draw[black, thick] (rv) -- (b\i);
    }
    \draw[black, thick] (rv) .. controls ($(rv)!0.5!(b2) + (0,0.2)$) .. (b2);

\end{tikzpicture}
\end{center}
\caption{A threshold graph with sixteen vertices.}
\label{fig: threshold example}
\end{figure}
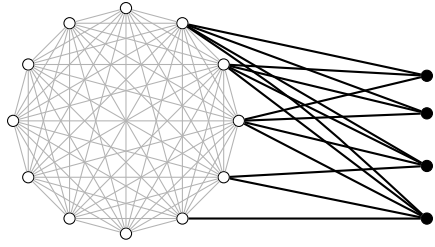

The class of threshold graphs is a subclass of both cographs and chordal graphs that was introduced by Chv\'atal and Hammer \cite{chvatal}, and independently by Henderson and Zalcstein \cite{Henderson1977}, in 1977. There are several known characterizations of threshold graphs. For instance, threshold graphs are exactly the graphs that can be obtained from the single-vertex graph by repeatedly adding either an isolated vertex or a dominating vertex. Threshold graphs are also exactly the $\{ C_4, P_4, 2K_2\}$-free graphs. For a comprehensive resource on threshold graphs, we refer the reader to the book by Mahadev and Peled \cite{Mahadev1995}.

In Section \ref{sec: d} we prove a structure theorem for any ideally connected graph with a minimum vertex cutset that forms a \textit{clique}, a set of pairwise adjacent vertices. In Section \ref{sec: chordal}, by using the property that every minimum vertex cutset of a chordal graph is a clique, we prove that the ideally connected chordal graphs are precisely the threshold graphs.

\begin{theorem}\label{thm: ic chordal iff threshold}
     Let $G$ be a chordal graph. Then $G$ is ideally connected if and only if $G$ is a threshold graph.
\end{theorem}

Theorem \ref{thm: ic chordal iff threshold} gives a new way of describing threshold graphs as a subclass of chordal graphs.

One of the most common ways to represent the structure of a chordal graph $G$ is through a \textit{clique tree}, a tree that has maximal cliques in $G$ as its nodes and where for any $v \in V(G)$, the subgraph of the tree induced by the nodes containing $v$ is connected. To further describe the subclass of chordal graphs formed by the threshold graphs, in Section \ref{section: clique trees} we characterize the clique trees of threshold graphs. We show that for any threshold graph $G$ with set $\mathcal{M}_G$ of maximal cliques, any tree on $|\mathcal{M}_G|$ nodes is a clique tree of $G$. This universality property is interesting because it does not hold for general chordal graphs or even for general split graphs.

\subsection{Notation}
We now describe the notation that is used throughout the paper. For a natural number $n$ we let $[n]$ denote the set $\{ 1,2,\ldots, n\}$. For a vertex $v$ in a graph $G$, the \textit{neighborhood of $v$}, denoted by $N_G(v)$, is the set of vertices in $G$ that are adjacent to $v$. The \textit{closed neighborhood of $v$}, denoted by $N_G[v]$, is the set $N_G(v) \cup \{v\}$. We let $P_k$ denote the graph isomorphic to a path on $k$ vertices.
For a path $P$, the \textit{length} of $P$ is the number of edges in $P$. 
For a path $P$ from vertex $u$ to vertex $v$, if $x$ and $y$ are internal vertices of $P$, then $xPy$ denotes the subpath of $P$ from $x$ to $y$. We say that a collection $\mathcal{P}$ of internally disjoint $u$-$v$ paths \textit{saturates $u$} if every vertex of $N_G(u) \setminus \{v\}$ is an internal vertex of exactly one path in $\mathcal{P}$. In particular, in any ideally connected graph $G$ and distinct vertices $u,v$ of $G$ with $\degr{G}{u} \leq \degr{G}{v}$, there exists a collection of $\degr{G}{u}$ internally disjoint $u$-$v$ paths saturating $u$. 
If $G$ is a graph with induced subgraphs $G_1$, $G_2$ which both have the subgraph $H$ as an induced subgraph, and moreover we have $G=G[V(G_1) \cup V(G_2)]$ and $V(G_1)\cap V(G_2)=V(H)$, then we say that $G$ arises from $G_1$ and $G_2$ by \textit{gluing along $H$}. 

For vertex-disjoint graphs $G_1$ and $G_2$, the \textit{graph union} $G_1 \cup G_2$ is the graph with vertex set $V(G_1) \cup V(G_2)$ and edge set $\{ uv: uv \in E(G_1)\cup E(G_2)\}$. Also, the \textit{graph join} $G_1 \nabla G_2$ of $G_1$ and $G_2$ is the graph obtained from $G_1 \cup G_2$ by adding edges $\{ uv: u \in V(G_1), v \in V(G_2)\}$.

For a subset $U \subseteq V(G)$, we let $G - U$ denote the subgraph of $G$ induced by the vertices in $V(G)\setminus U$. For a vertex $v \in V(G)$, we let $G-v$ denote the subgraph $G-\{ v\}$. For a connected graph $G$ that is not isomorphic to a complete graph and a vertex cutset $S$ of minimum size in $G$, the graph $G-S$ consists of multiple components. If $C_1, \ldots, C_m$ are the components of $G-S$, then the subgraphs $G[V(C_1) \cup S], \ldots, G[V(C_m) \cup S]$ will be called the \textit{$S$-subgraphs} of $G$.
If $S$ is a vertex cutset whose vertices form a clique in a connected graph $G$, then we say that $S$ is a \textit{clique cut} in $G$. If additionally we have $|S|=\kappa(G)$ and $G$ itself is not isomorphic to a complete graph, then we say that $S$ is a \textit{$\kappa$-clique cut} of $G$. Finally, a vertex $v \in V(G)$ is called \textit{simplicial} if $N_G(v)$ is a clique in $G$, in which case we say that $v$ is $\textit{simplicial to the clique $N_G(v)$}$.

\section{Ideally Connected Cographs}\label{sec: cograph}
In this section we characterize the ideally connected cographs. 
If $G_1$ and $G_2$ are vertex-disjoint cographs, one can show that $G_1 \cup G_2$ and $G_1 \nabla G_2$ are also cographs. In fact, Lerchs  \cite{lerchs} showed that cographs can be characterized as the graphs that can be constructed from isolated vertices by a finite number of graph union and graph join operations. Hence, to understand whether a cograph is ideally connected, we first study whether ideal connectedness is preserved when taking graph joins.

\begin{lemma}\label{lemma: join} Let $G_1$ and $G_2$ be vertex-disjoint cographs. Then $G_1 \nabla G_2$ is ideally connected if and only if both $G_1$ and $G_2$ are ideally connected.
\end{lemma}
\begin{proof}
    Let $G_1$ and $G_2$ be vertex-disjoint cographs and let $G$ be the graph join $G_1 \nabla G_2$.

First suppose $G_1$ and $G_2$ are both ideally connected. We will show that $G$ is ideally connected too.
    For distinct vertices $u,u' \in V(G_1)$, we may assume without loss of generality that $\degr{G}{u} \leq \degr{G}{u'}$. Clearly, $\degr{G}{u} = |N_{G_1}(u)| + |V(G_2)|$ and $\degr{G_1}{u} \leq \degr{G_1}{u'}$. Since $G_1$ is ideally connected, there is a collection $\mathcal{C}_1$ of $\degr{G_1}{u}$ internally disjoint $u$-$u'$ paths in $G_1$. Let $\mathcal{C}_2$ denote the collection of the $|V(G_2)|$ internally disjoint $u$-$u'$ paths of the form $\{ uvu': v \in V(G_2)\}$. Each path in $\mathcal{C}_2$ is internally disjoint from $V(G_1)$, so the paths in $\mathcal{C}_1 \cup \mathcal{C}_2$ show that $\kappa_G(u,u')=\mingen{G}{u}{u'}$. Similarly, for distinct vertices $v,v' \in V(G_2)$ we have $\kappa_{G}(v,v')=\mingen{G}{v}{v'}$.

    For vertices $u \in V(G_1)$ and $v \in V(G_2)$, we may assume without loss of generality that $\degr{G}{u} \leq \degr{G}{v}$. Let $\mathcal{C}_1$ denote the collection of paths consisting of the length-one path $uv$, the $|N_{G_1}(u)|$ paths of the form $\{ uyv: y \in N_{G_1}(u)\}$, and the $\degr{G_2}{v}$ paths of the form $\{ uwv: w \in N_{G_2}(v)\}$. Then $\mathcal{C}_1$ consists of $\degr{G_1}{u}+|N_{G_2}[v]|$ internally disjoint $u$-$v$ paths.
    Since $\degr{G}{u} \leq \degr{G}{v}$, we have $\degr{G_1}{u}+|V(G_2)| \leq \degr{G_2}{v}+|V(G_1)|$, which implies that $|V(G_1)|-\degr{G_1}{u} \geq |V(G_2)|-\degr{G_2}{v}$, so the number of nonneighbors of $u$ in $G_1$ is at least as large as the number of nonneighbors of $v$ in $G_2$. Hence, in $G$ there is a collection $\mathcal{C}_2$ of $|V(G_2) \setminus N_{G_2}[v]|$ internally disjoint $u$-$v$ paths of the form $\{ uxyv: x \in V(G_2) \setminus N_{G_2}[v], y \in V(G_1) \setminus N_{G_1}[u] \}$. Then, the paths in $\mathcal{C}_1 \cup \mathcal{C}_2$ show that $\kappa_G(u,v)=\degr{G_1}{u}+|V(G_2)|=\mingen{G}{u}{v}$. Hence, if $G_1$ and $G_2$ are ideally connected cographs, then so is $G$.

    Conversely, suppose that $G$ is ideally connected. 
    We will show that $G_1$ and $G_2$ are ideally connected. Suppose for the sake of contradiction that this is not the case, so without loss of generality we may assume that $G_1$ is not ideally connected. Then there are distinct vertices $u,u' \in V(G_1)$ with $\degr{G_1}{u} \leq \degr{G_1}{u'}$ and $\kappa_{G_1}(u,u') < \degr{G_1}{u}$. Clearly we have $\degr{G}{u}=\degr{G_1}{u}+|V(G_2)|<\degr{G}{u'}$. By the assumption that $G$ is ideally connected, there are $\degr{G}{u}$ internally disjoint $u$-$u'$ paths, saturating each neighbor of $u$ in $G$. Hence, since $N_G(u)$ contains $V(G_2)$, there are $\degr{G}{u}-|V(G_2)|=\degr{G_1}{u}$ internally disjoint $u$-$u'$ paths in $G_1$, a contradiction. Hence, if $G$ is ideally connected, then $G_1$ and $G_2$ must be ideally connected too.
\end{proof}

We can now prove Theorem \ref{thm: intro cograph}, which states that a cograph is ideally connected if and only if it is $2K_2$-free.

\begin{proof}[Proof of Theorem \ref{thm: intro cograph}]
Let $G$ be a cograph. We will prove the theorem by induction on $|V(G)|$. The theorem clearly holds for $|V(G)| \leq 2$ and clearly holds if $G$ is an edgeless graph, so we may assume that $|V(G)| \geq 3$ and $G$ has at least one edge.

First suppose that $G$ is ideally connected. Since $G$ is a cograph, there are vertex-disjoint cographs $G_1$ and $G_2$ so that either $G = G_1 \cup G_2$ or $G=G_1 \nabla G_2$. If $G=G_1 \cup G_2$, then since $G$ is ideally connected, both $G_1$ and $G_2$ are ideally connected and at least one of them is edgeless. We assume without loss of generality that $G_2$ is edgeless. Since $G_1$ is ideally connected, the inductive hypothesis implies that $G_1$ is $2K_2$-free. Clearly, the graph union of a $2K_2$-free graph with an edgeless graph is $2K_2$-free, so in this case $G$ is $2K_2$-free. On the other hand, if $G = G_1 \nabla G_2$, then by Lemma \ref{lemma: join}, both $G_1$ and $G_2$ are ideally connected. By the inductive hypothesis, $G_1$ and $G_2$ are $2K_2$-free. It is easy to see that the graph join of two vertex-disjoint $2K_2$-free graphs is also $2K_2$-free, so in this case $G$ is also $2K_2$-free. Hence, if $G$ is an ideally connected cograph, then $G$ is $2K_2$-free.

Now suppose $G$ is  a $2K_2$-free cograph. Then there are vertex-disjoint $2K_2$-free cographs $G_1$ and $G_2$ so that either $G = G_1 \cup G_2$ or $G=G_1 \nabla G_2$. By the inductive hypothesis, both $G_1$ and $G_2$ are ideally connected. If $G = G_1 \nabla G_2$, then $G$ is ideally connected by Lemma \ref{lemma: join}. If $G = G_1 \cup G_2$ then since $G$ is $2K_2$-free, at least one of $G_1$ and $G_2$ is edgeless. In this case, since the graph union of an ideally connected graph with an edgeless graph is also ideally connected, then $G$ is ideally connected. Hence, any $2K_2$-free cograph is ideally connected.
\end{proof}

For vertex-disjoint cographs $G_1$ and $G_2$, say that the graph union operation $G_1 \cup G_2$ is \textit{safe} if either $G_1$ or $G_2$ is edgeless. It follows from Theorem \ref{thm: intro cograph} that the ideally connected cographs are exactly the graphs that can be constructed from isolated vertices by a finite number of graph join and safe graph union operations. 

Theorem \ref{thm: intro cograph} shows that the ideally connected cographs are exactly the $\{P_4, 2K_2\}$-free graphs.
The complement of the graph $P_4$ is isomorphic to $P_4$ and the complement of the graph $2K_2$ is the cycle $C_4$. The $\{P_4, C_4\}$-free graphs are known as \textit{trivially perfect graphs}, so Theorem \ref{thm: intro cograph} shows that the ideally connected cographs are exactly the complements of trivially perfect graphs.
Since the threshold graphs are the $\{ P_4, C_4, 2K_2\}$-free graphs, the threshold graphs are exactly the ideally connected cographs whose complements are also ideally connected.
In particular, Theorem \ref{thm: intro cograph} shows that threshold graphs are ideally connected. This fact will be used in later sections.
\begin{corollary}\label{cor: threshold ic} If $G$ is a threshold graph, then $G$ is ideally connected. \qed
\end{corollary}

An interesting direction for future work would be to determine the ideally connected graphs in graph classes that are superclasses of cographs. In particular, an approach extending the one taken in this section seems promising for graph classes defined by split decompositions (edge cuts that induce complete bipartite graphs), such as distance-hereditary graphs, parity graphs, or circle graphs.
\section{Ideally Connected Graphs with $\kappa$-Clique Cuts} \label{sec: d}
In this section we analyze the structure of ideally connected graphs with $\kappa$-clique cuts, which we will apply to chordal graphs in Section \ref{sec: chordal}. We begin with a few simple lemmas about the degrees of pairs of vertices in an ideally connected graph.
\begin{lemma} \label{lem: subg} Let $G$ be a graph with a $\kappa$-clique cut $S$. Let $t=|S|$. If $G$ is ideally connected, then at most one component of $G-S$ has a vertex of degree greater than $t$ in $G$. 
\end{lemma} 
\begin{proof} Suppose that $C_i$ and $C_j$ are distinct components of $G-S$ and there is a vertex $v \in V(C_i)$ with $\degr{G}{v} > t$. For any vertex $v'$ in $V(C_j)$, we will show that $\degr{G}{v'} \leq t$.

Any $v$-$v'$ path intersects $S$, so  $\kappa_G(v,v') \leq |S|=t$. Since $G$ is ideally connected, we have $\mingen{G}{v}{v'} = \kappa_G(v,v') \leq t$. Then, since $\degr{G}{v} > t$, we must have $\degr{G}{v'} \leq t$.
\end{proof}

The lemma shows that in an ideally connected graph $G$ with a $\kappa$-clique cut $S$, for any pair of vertices $u,v \in V(G) \setminus S$ in distinct $S$-subgraphs of $G$, we have $\mingen{G}{u}{v}=|S|$.

The next lemma considers the quantity $\mingen{G}{u}{s}$ for vertices $u \in V(G) \setminus S$ and $s \in S$ in an ideally connected graph $G$ with a $\kappa$-clique cut $S$.

\begin{lemma}\label{lemma: u-s} Let $G$ be a graph with a $\kappa$-clique cut $S$, and suppose $u \in V(H) \setminus S$ for an $S$-subgraph $H$ of $G$. If $G$ is ideally connected, then for every $s \in S$, we have $\degr{H}{u}= \degr{G}{u} \leq \degr{H}{s} < \degr{G}{s}$ and $\kappa_H(u,s)=\kappa_G(u,s)=\degr{H}{u}$.
\end{lemma}
\begin{proof}
    The identity $\degr{H}{u}=\degr{G}{u}$ is obvious. By definition of a $\kappa$-clique cut, $G-S$ has at least two components and so the vertex $s$ has a neighbor $z$ in a component of $G-S$ different from $H-S$. Then, $\degr{H}{s}<\degr{G}{s}$. Hence, it remains to show that $\degr{G}{u}\leq \degr{G}{s}$.

    Suppose for the sake of contradiction that $\degr{G}{u}> \degr{G}{s}$. Since $G$ is ideally connected, there's a collection $\mathcal{P}=\{ P_1, \ldots, P_{\degr{G}{s}}\}$ of $\degr{G}{s}$ internally disjoint $u$-$s$ paths in $G$ saturating $s$, so that every neighbor of $s$ appears in some path in $\mathcal{P}$. Let $P_i \in \mathcal{P}$ denote the path in this collection that contains the vertex $z$, so $P_i$ is a $s$-$u$ path that includes the vertex $z$ but no vertex in $S\setminus \{ s\}$. However, in order to reach vertex $u$ from $z$, the path $P_i$ must intersect $S \setminus \{ s\}$, a contradiction. This contradiction establishes that $\degr{G}{u}\leq \degr{G}{s}$, and so since $G$ is ideally connected, we have $\kappa_G(u,s)=\degr{H}{u}=\kappa_H(u,s)$.
\end{proof}

The lemma implies the following corollary.

\begin{corollary} If $G$ is an ideally connected graph with a $\kappa$-clique cut $S$, then $S$ is the unique $\kappa$-clique cut in $G$. Moreover, we have 
\begin{center}
    $ \Delta(G) = \max\{\operatorname{deg}_G(v): v \in S\} \geq |S|+1$ 
\end{center}
\end{corollary}
\begin{proof}
    The second property follows immediately from Lemma \ref{lemma: u-s}. To prove the first property, suppose for the sake of contradiction that $S'$ is another $\kappa$-clique cut in $G$ that is different from $S$. Then, since both $S$ and $S'$ have size $\kappa(G)$, there are vertices $s \in S\setminus S'$ and $s' \in S'\setminus S$. Let $H$ denote the $S$-subgraph of $G$ containing $s'$ and let $H'$ denote the $S'$-subgraph of $G$ containing $s$. By applying Lemma \ref{lemma: u-s} to the $S'$-subgraph $H'$, we have $\degr{G}{s}=\degr{H'}{s} \leq \degr{H'}{s'}<\degr{G}{s'}$. However, by applying the Lemma \ref{lemma: u-s} to the $S$-subgraph $H$ we obtain $\degr{G}{s'}=\degr{H}{s'} \leq \degr{H}{s}<\degr{G}{s}$, which contradicts that $\degr{G}{s}<\degr{G}{s'}$. Hence, $S$ must be the unique $\kappa$-clique cut in $G$.
\end{proof}


Finally, we compare the degrees of vertices $s_1, s_2 \in S$ in an ideally connected graph $G$ with a $\kappa$-clique cut $S$.

\begin{lemma}\label{lem s-s} Let $G$ be a graph with a $\kappa$-clique cut $S$, and let $s_1, s_2$ be distinct vertices in $S$ where $\degr{G}{s_1} \leq \degr{G}{s_2}$. Then for any $S$-subgraph $H$ of $G$, we have $\mingen{H}{s_1}{s_2}=\degr{H}{s_1}$ and $\kappa_H(u,v)=\degr{H}{s_1}$.
\end{lemma}
\begin{proof}
    Let $c = \degr{G}{s_1}$ and let $t=|S|$. Since $G$ is ideally connected, there's a collection  $\mathcal{P} = \{ P_1, \ldots, P_c\}$ of $c$ internally disjoint $s_1$-$s_2$ paths in $G$ saturating $s_1$. Without loss of generality, we may assume $P_1$ is the length-one path $s_1s_2$ and that $\{ P_2, \ldots, P_{|S|-1}\}$ are the distinct length-two paths of the form $s_1s's_2$ where $s' \in S \setminus \{ s_1, s_2\}$. Hence, there are $c-(t-1) > 0$ $s_1$-$s_2$ paths, each of length greater than one, in $G$ that are internally disjoint from one another and internally disjoint from $S$. Any such path has its interior entirely in a single component of $G-S$, so since the paths in $\mathcal{P}$ saturate $s_1$, the vertex $s_2$ must have at least at many neighbors as $s_1$ in any component of $G-S$. Hence, in any $S$-subgraph $H$ of $G$, we have $\mingen{H}{s_1}{s_2}=\degr{H}{s_1}$ and $\kappa_H(s_1,s_2)=\degr{H}{s_1}$. 
\end{proof}

If $G$ is an ideally connected graph, it is natural to ask for sufficient conditions for a subgraph of $G$ to also be ideally connected. When $G$ has a $\kappa$-clique cut $S$, we now show that each $S$-subgraph of $G$ is also ideally connected, which gives one sufficient condition for ideal connectedness to be preserved when taking subgraphs. 

\begin{lemma} \label{lem: complete cutset} Let $S$ be a $\kappa$-clique cut of an ideally connected graph $G$, and let $t=|S|$. Then for any $S$-subgraph $H$ of $G$, $H$ is ideally connected and $\kappa(H) \geq t$.
\end{lemma}
\begin{proof} Let $H$ be an $S$-subgraph of $G$. We first show that $H$ is ideally connected. Let $u,v \in V(H)$ be distinct vertices of $H$ and let $c=\mingen{H}{u}{v}$. We will show that $\kappa_H(u,v) = c$.

If $u \in V(H) \setminus S$ and $v \in S$ then $\kappa_H(u,v)=\kappa_G(u,v)=\degr{G}{u}=c$ by Lemma \ref{lemma: u-s}. If $u,v \in S$ are distinct vertices in $S$, then $\kappa_H(u,v)=\mingen{H}{u}{v}=c$ by Lemma \ref{lem s-s}.  Hence, it remains to show that for distinct vertices $u,v \in V(H) \setminus S$, we have $\kappa_H(u,v)=\mingen{H}{u}{v}$. Since $\degr{H}{u}=\degr{G}{u}$ and $\degr{H}{v}=\degr{G}{v}$ for $u,v \in V(H)\setminus S$, we must show that $\kappa_H(u,v)=\mingen{G}{u}{v}$.

For $u,v \in V(H) \setminus S$, let $\tilde{\mathcal{P}}$ denote the set of all $c$-tuples $(P_1, \ldots, P_c)$ where $\{ P_1, \ldots, P_c\}$ is a set of $c$ internally disjoint $u$-$v$ paths. Since $G$ is ideally connected, $\tilde{\mathcal{P}}$ is nonempty. Among all tuples in $\tilde{\mathcal{P}}$, choose $\mathcal{P} = (P^*_1, \ldots, P^*_c) \in \tilde{\mathcal{P}}$ so that the corresponding $c$-tuple of lengths $(|P^*_1|, \ldots, |P^*_c|)$ has minimum lexicographical order among all the tuples of lengths from elements of $\tilde{\mathcal{P}}$. The minimality of $\mathcal{P}$ implies that in particular, $|P^*_i| \leq |P^*_j|$ for all $1 \leq i \leq j \leq c$, and $P^*_1$ is a path of length one if and only if $uv \in E(G)$. 

We claim that for each $i \in [c]$, the path $P^*_i$ is a subgraph of $H$. For the sake of contradiction, suppose that this is not the case, meaning that for some $i \in [c]$, the path $P^*_i$ contains a vertex $ z \in V(G) \setminus V(H)$. Then since $S$ is a vertex cutset, there is a subpath $P'_i$ of $P^*_i$ and distinct vertices $s_1, s_2 \in S$ such that $P'_i=s_1P^*_is_2$ and $P'_i$ contains the vertex $z$. Then replacing $P'_i$ with the edge $s_1s_2$ on $P^*_i$, we obtain a new $u$-$v$ path $Q^*_i$ of length strictly less than $|P^*_i|$. The path $Q^*_i$ remains internally disjoint from each path in $\{P^*_1, \ldots, P^*_{i-1}, P^*_{i+1}, \ldots, P^*_c\}$ because $V(Q^*_i) \subseteq V(P^*_i)$. Also, $Q^*_i$ has length greater than one because it still includes $s_1$ and $s_2$ as vertices, so $|V(Q^*_i)| \geq 4$. Since $Q^*_i$ has length greater than one, it is distinct from the path $uv$ if such path exists, so it is distinct from each of the paths in $\{P^*_1, \ldots, P^*_{i-1}, P^*_{i+1}, \ldots, P^*_c\}$. Then $(P^*_1, \ldots, P^*_{i-1}, Q^*_i, P^*_{i+1}, \ldots, P^*_c) \in \tilde{\mathcal{P}}$, contradicting the choice of $\mathcal{P}$. This contradiction establishes that for each $i \in [c]$, the path $P^*_i$ is a subgraph of $H$.

Therefore, the collection $\{ P^*_1, \ldots, P^*_c\}$ is a set of $c=\mingen{H}{u}{v}$ internally disjoint $u$-$v$ paths in $H$, so $\kappa_H(u,v) =c$ and $H$ is ideally connected.

    Now we show $\kappa(H) \geq t$. Since $H$ is ideally connected it suffices to show that minimum degree $\delta(H)$ of $H$ is at least $t$. Since $\kappa(G)=t$, we have $\delta(G) \geq t$, so for any vertex $v \in V(H) \setminus S$, we have $\degr{H}{v}=\degr{G}{v}\geq t$. For a vertex $s \in S$, we have $\degr{H}{s} \geq t$ because $s$ has $t-1$ neighbors in $G[S]$ and by the minimality of the cut $S$, $s$ has at least one neighbor in $H-S$. Hence, $\kappa(H) \geq t$. 
\end{proof}

The conclusion of Lemma \ref{lem: complete cutset} does not hold without the assumption that $G$ has a minimum cutset that is a clique. Indeed, if $S$ is a minimum vertex cut of an ideally connected graph $G$ and $C$ is a component of $G - S$, there may be vertices $u,v \in V(C)$ with a path between them passing through another component of $G-S$. Crucially, if $S$ is not a clique, there may be $u$-$v$ paths leaving $G[V(C) \cup S]$ that cannot be rerouted through $S$, so the subgraph $G[V(C) \cup S]$ might not be ideally connected. It is for this reason that ideal connectedness is harder to analyze in non-chordal graphs, since as we will see in Section \ref{sec: chordal}, such graphs may not have a $\kappa$-clique cut.

We can now give a structure theorem for ideally connected graphs with $\kappa$-clique cuts. The theorem shows that if $G$ is an ideally connected graph with $\kappa$-clique cut $S$, then each $S$-subgraph of $G$ is ideally connected and any two $S$-subgraphs intersect precisely at $S$. Moreover, $G$ is the union of its $S$-subgraphs and satisfies certain degree restrictions.
\begin{theorem} \label{thm: conditions}  Suppose $G$ is a graph with a $\kappa$-clique cut $S$. Let $t = |S|$, let $H_1, \ldots, H_m$ denote the $S$-subgraphs of $G$, and for each $i \in [m]$, let $C_i$ be the component of $G-S$ given by $H_i-S$. 

Then $G$ is ideally connected if and only if each of the following conditions hold:
\begin{theoremenum}
    \item \label{item1} Every $S$-subgraph of $G$ is an ideally connected, $t$-connected graph.
    \item \label{item2} Among each of the subgraphs $C_1, \ldots, C_m$, at most one has a vertex of degree greater than $t$ in $G$. If such a component, say $C_j$, exists, then $\degr{H_j}{u} \leq \degr{H_j}{s} < \degr{G}{s}$ for every $u \in V(H_j)\setminus S$ and $s \in S$. 
    \item  \label{item3} For distinct vertices $s_1, s_2 \in S$ where $\degr{G}{s_1} \leq \degr{G}{s_2}$, we have $\degr{H_i}{s_1} \leq \degr{H_i}{s_2}$ for every $i \in [m]$.
\end{theoremenum}
\end{theorem}
\begin{proof} [Proof of Theorem \ref{thm: conditions}]
    If $G$ is ideally connected, then Conditions \ref{item1}, \ref{item2}, and \ref{item3} follow from Lemmas \ref{lem: subg}, \ref{lemma: u-s}, \ref{lem s-s}, and $\ref{lem: complete cutset}$.

    Conversely, suppose $G$ is a graph with a $\kappa$-clique cut $S$ satisfying Conditions \ref{item1}, \ref{item2}, and \ref{item3} in the theorem statement. We will show that $G$ is ideally connected. Without loss of generality, we may assume $C_1$ is the only component of $G-S$ that might have a vertex of degree larger than $t$ in $G$. 
    
    For any pair of vertices $u,v \in V(G)$ with $v \in V(G) \setminus V(H_1)$, we have $\mingen{G}{u}{v}=t$ by Condition \ref{item1}. Then, since $G$ is $t$-connected, we have $\kappa_G(u,v)=t=\mingen{G}{u}{v}$, as needed.

    Next, for vertices $u \in V(C_1)$ and $s \in S$, Condition \ref{item2} implies that $\mingen{G}{u}{s}=\degr{H_1}{u}$. By Condition \ref{item1}, $H_1$ is ideally connected, so $\kappa_{H_1}(u,s)=\degr{H_1}{u}$. Since $\degr{G}{u}=\degr{H_1}{u}$, we have $\kappa_G(u,s)=\degr{G}{u}$, as needed.

    For distinct vertices $s_1, s_2 \in S$, we may assume without loss of generality that $\degr{G}{s_1} \leq \degr{G}{s_2}$. By Condition \ref{item1}, each $H_i$ is ideally connected, and by Condition \ref{item3}, we have $\degr{H_i}{s_1} \leq \degr{H_i}{s_2}$ for all $i \in [m]$. Then, for each $i \in [m]$, there is a collection of $\degr{H_i}{s_1}$ internally disjoint $s_1$-$s_2$ paths saturating $s_1$. We may assume that among these $\degr{H_i}{s_1}$ paths, exactly one is the length-one path $s_1s_2$ and exactly $t-2$ are the length-two paths of the form $s_1s's_2$ for $s' \in S \setminus \{ s_1, s_2\}$, and so the others are internally disjoint from each other and internally disjoint from $S$. Thus, for each $i \in [m]$, there are $\degr{H_i}{s_1}-(t-1)$ internally disjoint $s_1$-$s_2$ paths that have their (nonempty) interiors entirely in $C_i$. Then, \[ \kappa_G(u,v) \geq (t-1)+\sum_{i=1}^m (\degr{H_i}{s_1}-(t-1))=\degr{G}{s_1}\] as needed.

Finally, for distinct vertices $u,v \in V(H_1) \setminus S$, we have $\degr{H_1}{u}=\degr{G}{u}$ and $\degr{H_1}{v}=\degr{G}{v}$. Since $H_1$ is ideally connected, we have $\kappa_{G}(u,v) \geq \kappa_{H_1}(u,v)=\mingen{H_1}{u}{v}$. Hence, $\kappa_G(u,v)=\mingen{G}{u}{v}$.

Therefore, we have $\kappa_G(u,v) = \mingen{G}{u}{v}$ for any distinct vertices $u,v \in V(G)$, so $G$ is ideally connected.
\end{proof}

For instance, the theorem implies that the ideally connected graphs of connectivity $1$ on $n$ vertices can be obtained as follows: take an ideally connected graph $H$ on $|V(H)|<n$ vertices, choose a vertex $v$ of largest degree in $H$, and then add $n-|V(H)|$ vertices of degree one, each adjacent only to $v$. Similarly, one can view ideally connected graphs of connectivity two as arising from an ideally connected, $2$-connected graph $H$ by adding  internally disjoint $u$-$v$ paths of length at least two to $H$ where $u,v \in V(H)$ are adjacent vertices in $H$ and have largest degree in $H$. See Figure \ref{fig: k-clique cut}.
In general, Theorem \ref{thm: conditions} shows that any ideally connected graph with a $\kappa$-clique cut $S$ of size $t$ can be obtained by taking $t$-connected graphs $H_1, \ldots, H_m$, which each contain a clique isomorphic to $S$, and gluing them together along the clique $S$ in a way that satisfies the degree conditions of Theorem \ref{thm: conditions}.

\begin{figure}[h]
\begin{center}
\begin{tikzpicture}[scale=1.2,
    every node/.style={draw, circle, fill=white, inner sep=1.5pt},
    node distance=1cm
]

\begin{scope}[shift={(-0.5,0)}]
\node (c1) at (-3, 0.8) {};
\node (c2) at (-3, -0.8) {};
\node (s1) at (-2, 1) {$s_1$};
\node (s2) at (-2, -1) {$s_2$};
\node (v1) at (-1, 1.5) {};
\node (v2) at (0, 1.2) {};
\node (a1) at (-1, 0.6) {};
\node (a2) at (0.3, -0.2) {};
\node (a3) at (-0.4, 0.3) {};
\node (b1) at (0, -0.8) {};

\draw (c1) -- (s1);
\draw (c1) -- (c2);
\draw (c1) -- (s2);
\draw (c2) -- (s1);
\draw (c2) -- (s2);
\draw (s1) -- (s2);
\draw (s1) -- (v1);
\draw (v1) -- (v2);
\draw (v2) -- (s2);
\draw (s1) -- (a1);
\draw (a1) -- (a3);
\draw (a2) -- (a3);
\draw (a2) -- (s2);
\draw (s1) -- (b1);
\draw (s2) -- (b1);


\end{scope}

\begin{scope}[shift={(1,0)}]
\filldraw[color=red!60, fill=red!5, very thick](1, -1) rectangle (3.5, 1);
\node (qs1) at (3.5, 1) {$s_1$};
\node (qs2) at (3.5, -1) {$s_2$};

\node[draw=none, fill opacity=0, text opacity=1] at (2.25, 0.5) {$H_1$};
\node[draw=none, fill opacity=0, text opacity=1] at (2.25, 0) {$\text{ideally connected}$};

\node (as1) at (4.5, 2) {$s_1$};
\node (as2) at (4.5, 1.4) {$s_2$};
\node (bs1) at (4.5, 0.4) {$s_1$};
\node (bs2) at (4.5, -0.2) {$s_2$};
\node (cs1) at (4.5, -1.2) {$s_1$};
\node (cs2) at (4.5, -1.8) {$s_2$};

\draw (qs1) -- (qs2);
\draw [blue](as1) -- (as2);
\draw [blue](bs1) -- (bs2);
\draw [blue](cs1) -- (cs2);
\draw [blue] (as1) .. controls +(2,0) .. (as2);
\draw [blue] (bs1) .. controls +(3,0) .. (bs2);
\draw [blue] (cs1) .. controls +(1,0) .. (cs2);

\node[rectangle, text=blue, draw=none, minimum size=1pt] at (6, 1.3) {$H_2$};
\node[rectangle, text=blue, draw=none, minimum size=1pt] at (6, -0.2) {$H_3$};
\node[rectangle, text=blue, draw=none, minimum size=1pt] at (6, -1.4) {$H_4$};

\end{scope}
\end{tikzpicture}
\end{center}
\caption{An example showing how an ideally connected graph with a $\kappa$-clique cut $S$ is obtained by gluing together smaller ideally connected graphs along $S$. On the left is an ideally connected graph $G$ with a $\kappa$-clique cut $S = \{ s_1, s_2\}$. On the right, we see how $G$ arises by gluing together ideally connected graphs along a clique of size two, according to Theorem \ref{thm: conditions}. The graphs $H_1, \ldots, H_4$ are the $S$-subgraphs of $G$, and $H_1$ is ideally connected and of connectivity at least two while $H_2$, $H_3$, and $H_4$ are $2$-regular and ideally connected, so are cycles. }
\label{fig: k-clique cut}
\end{figure}
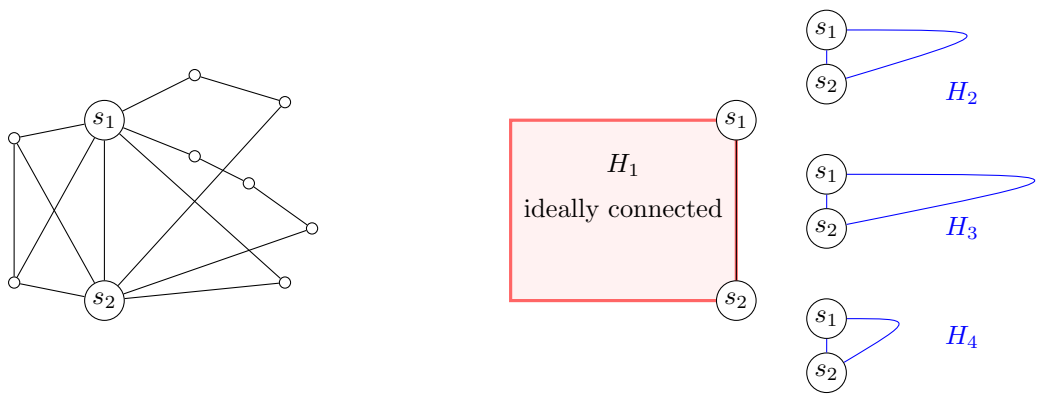

Conversely, if $G$ is a graph obtained by gluing together certain ideally connected graphs, it is natural to ask for a sufficient condition for $G$ itself to be ideally connected. When $G$ is obtained by gluing ideally connected graphs along a clique, satisfying the degree conditions of Theorem \ref{thm: conditions}, then the resulting graph is ideally connected. Figure \ref{fig: k-clique cut} illustrates this for a graph of connectivity two.

\section{Ideally Connected Chordal Graphs} \label{sec: chordal}
In this section we apply Theorem \ref{thm: conditions} to characterize the ideally connected chordal graphs. Recall that a graph is called \textit{chordal} if it has no induced cycle of length larger than three. There are several known characterizations of chordal graphs, one of which is the following. 
\begin{theorem} [Dirac \cite{Dirac1961}] \label{dirac} Let $G$ be a graph not isomorphic to a complete graph. Then $G$ is chordal if and only if every minimum vertex cutset in $G$ is a $\kappa$-clique cut. Moreover, if $G$ is chordal and $S$ is a $\kappa$-clique cut in $G$, then every $S$-subgraph of $G$ is chordal.
\end{theorem} 

Recall that a vertex $v$ in a chordal graph $G$ is called \textit{simplicial} if its neighborhood $N_G(v)$ is a clique in $G$. By applying Theorem \ref{thm: conditions}, we obtain the following lemma.

\begin{lemma} \label{chordal lem}Let $G$ be a chordal graph not isomorphic to a complete graph, and let $S$ be a $\kappa$-clique cut in $G$ with $t=|S|$. Then $G$ is ideally connected if and only if the $S$-subgraphs of $G$ can be listed as $H_1, \ldots, H_m$ such that the following properties hold:
\begin{theoremenum}
    \item $H_1$ is an ideally connected, $t$-connected chordal graph.\label{cond1}
    \item For each $i \in \{2,3,\ldots, m\}$, $H_i$ consists of the complete graph $G[S]$ along with a single vertex $v_i$ that is simplicial to $S$. \label{cond2}
\end{theoremenum}
\end{lemma}
\begin{proof} Let $G$ be a chordal graph not isomorphic to a complete graph and with a $\kappa$-clique cut $S$. First suppose $G$ is ideally connected. Then by Theorem \ref{thm: conditions}, we can list the $S$-subgraphs of $G$ as $H_1, \ldots, H_m$ such that the following conditions hold:
\begin{enumerate}
    \item $H_i$ is ideally connected and $\kappa(H_i) \geq t$ for each $i$. \label{a}
    \item For each $i \in \{ 2,3,\ldots, m\}$, we have $\kappa(H_i)=t$ and every vertex in $V(H_i)\setminus S$ has degree exactly $t$ in $G$. For any vertices $u \in V(H_1)\setminus S$ and $s \in S$, we have $\degr{H_1}{u} \leq \degr{H_1}{s}<\degr{G}{s}$. \label{c}
    \item For distinct vertices $s_1, s_2 \in S$ where $\degr{G}{s_1} \leq \degr{G}{s_2}$, we have $\degr{H_i}{s_1} \leq \degr{H_i}{s_2}$ for every $i \in [m]$. \label{d}
\end{enumerate} 

By Condition \ref{a} and Theorem \ref{dirac}, $H_1$ is ideally connected, $t$-connected, and chordal, which proves Property \ref{cond1}.
Now fix $i \in \{ 2,3,\ldots,c \}$. We claim that $H_i$ is isomorphic to the complete graph $K_{t+1}$. By Theorem \ref{dirac}, $H_i$ is chordal and so either $H_i$ is isomorphic to a complete graph or $H_i$ has a $\kappa$-clique cut. In the former case, Condition \ref{c} implies that $H_i$ must be isomorphic to $K_{t+1}$. 

We claim that the latter case cannot hold. Suppose for the sake of contradiction that for some $i \in \{ 2,3, \ldots, m\}$, the subgraph $H_i$ is not isomorphic to a complete graph. Then $H_i$ has a $\kappa$-clique cut $U$ of size $t$, so $H_i - U$ has at least two components. If $V(U) \setminus S$ contains a vertex $z$, then $z$ has $t-1$ neighbors in $U$ and has neighbors to each component of $H_i-U$, so $z$ has degree at least $t+1$ in $G$, contradicting Property \ref{c}. Hence, we must have $U \subseteq S$ and so since $|U|=t=|S|$, we must have $U=S$. But since $H_i-S$ is connected, the subgraph $H_i-U$ is also connected, which is a contradiction to the fact that $H_i-U$ has at least two components. 

Therefore, for each $i \in \{ 2,3, \ldots, m\}$, the subgraph $H_i$ is isomorphic to the complete graph $K_{t+1}$, so $H_i$ consists of $G[S]$ along with a single vertex $v_i$ that is simplicial to $S$. This proves Property \ref{cond2}.


Conversely, if $G$ is a chordal graph not isomorphic to a complete graph and $S$ is a $\kappa$-clique cut in $G$, and moreover $G$ satisfies Property \ref{cond1} and Property \ref{cond2}, then it is easy to see that $G$ satisfies the conditions from Theorem \ref{thm: conditions} that guarantee that $G$ is ideally connected.
\end{proof}

We can now prove Theorem \ref{thm: ic chordal iff threshold}, that the ideally connected chordal graphs are precisely the threshold graphs.

\begin{proof}[Proof of Theorem \ref{thm: ic chordal iff threshold}]
Let $G$ be a chordal graph. If $G$ is a threshold graph, then by Corollary \ref{cor: threshold ic}, $G$ is ideally connected.

Hence, it only remains to show that any ideally connected chordal graph is a threshold graph. Suppose that $G$ is ideally connected.
We will prove that $G$ is a threshold graph by induction on $|V(G)|$. Clearly, the theorem holds if $|V(G)| \leq 4$ or if $G$ is isomorphic to a complete graph, so we may assume $|V(G)| \geq 5$ and that  $G$ has a $\kappa$-clique cut $S$. 

Let $t=|S|$. By Lemma \ref{chordal lem}, the $S$-subgraphs of $G$ can be listed as $H_1, \ldots H_m$ such that the following properties hold:
\begin{enumerate}
    \item $H_1$ is an ideally connected, $t$-connected chordal graph.
    \item For each $i \in \{2,3,\ldots, m\}$, $H_i$ consists of $G[S]$ along with a single vertex $v_i$ that is simplicial to $S$. 
\end{enumerate}

If $H_1$ is isomorphic to a complete graph, then letting $K$ be the vertices in $V(H_1)$ and $I$ the vertices $v_2,v_3, \ldots, v_m$, we see that $G$ is $(K,I)$-split and all vertices in $I$ have the same neighborhood, so $G$ is a threshold graph.

Thus, we may assume that $H_1$ is not isomorphic to a complete graph, so by Theorem \ref{dirac}, $H_1$ has a $\kappa$-clique cut $S'$.

We claim that $V(S')$ contains $V(S)$. Suppose for the sake of contradiction that this is not the case, so  since $|S'| \geq t$, there are vertices $s \in S \setminus S'$ and $s' \in S' \setminus S$. Since $H_1$ is ideally connected and has $S'$ as a $\kappa$-clique cut, Lemma \ref{lemma: u-s} implies $\degr{H_1}{s} < \degr{H_1}{s'}$. However, since $S$ is a $\kappa$-clique cut in $G$ containing $s$ but not containing $s'$, Lemma \ref{lemma: u-s} also 
implies that $\degr{H_1}{s'} \leq \degr{H_1}{s}$, a contradiction. Hence, $V(S')$ contains $V(S)$.

Since $|V(H_1)| <|V(G)|$, we may apply the inductive hypothesis to $H_1$ to conclude that $H_1$ is a threshold graph. So, there is a clique $K'$ and independent set $I'$ in $H_1$ such that $H_1$ is $(K', I')$-split and the vertices of $I'$ have the nested neighborhood property in $H_1$. Since $H_1$ is not isomorphic to a complete graph, we can choose $K'$ and $I'$ so that for any vertex $v' \in I'$, the neighborhood $N_{H_1}(v')$ of $v'$ in $H_1$ is a proper subset of $K'$. In particular, we may choose $K'$ and $I'$ so that $S \subseteq S' \subseteq K'$. Hence, since $S'$ is a $\kappa$-clique cut of $H_1$, for every vertex $v' \in I'$ we have $S \subseteq S' \subseteq N_{H_1}(v')$. Then, taking $K=K'$ and adding the vertices $v_2, \ldots, v_m$ to $I'$ to produce the independent set $I$, we see that $G$ is $(K,I)$-split and the vertices of $I$ have the nested neighborhood property in $G$. Hence, $G$ is a threshold graph.

\end{proof}

Let $u,v \in V(G)$ be distinct vertices of an ideally connected chordal graph $G$ where $\degr{G}{u} \leq \degr{G}{v}$. From Theorem \ref{thm: ic chordal iff threshold}, we see that if $\degr{G}{u} < \degr{G}{v}$ then $N_G(u) \subseteq N_G[v]$, and if $\degr{G}{u} = \degr{G}{v}$ then either $N_G(u)=N_G(v)$ or $N_G[u]=N_G[v]$. 
As a result, for any distinct vertices $u,v$ in an ideally connected chordal graph $G$, we can not only conclude that $\mingen{G}{u}{v}$ internally disjoint $u$-$v$ paths exist in $G$, but we can also guarantee that each of these paths have length at most two, and therefore that these paths can be found very efficiently. Indeed, for any two vertices $u,v$ in a threshold graph $G$, the collection of length-two paths $\{ u zv: z\in N_G(u) \cap N_G(v)\}$, along with the length-one path $uv$ if it exists, gives a collection $\mingen{G}{u}{v}$ internally disjoint $u$-$v$ paths in $G$. This is related to the fact that recognizing if a graph is threshold can be done in linear time \cite{heggernes}.

\section{Clique Trees of Threshold Graphs}\label{section: clique trees}
In this section, we further describe threshold graphs within the class of chordal graphs by characterizing the clique trees of threshold graphs. 

We first state some terminology related to intersection graphs and clique trees.
For a family $\mathcal{C}= \{ S_1, \ldots, S_m \}$ of $m$ sets, the \textit{intersection graph} of $\mathcal{C}$ is the graph $G$ with vertices $V(G)=\{ v_1, \ldots, v_m\}$ and edges $E(G) = \{ v_iv_j: i \neq j \text { and } S_i \cap S_j \neq \emptyset \}$. Gavril \cite{Gavril1974} proved that a graph is chordal if and only if it is the intersection graph of a family of subtrees of a tree.
For a chordal graph $G$ and a tree $T$, let $T \in \mathcal{T}_G$ if $G$ can be represented as the intersection graph of a family of subtrees of $T$. Since $\mathcal{T}_G$ is infinite, it is useful to only consider the trees in $\mathcal{T}_G$ of minimum size, in which case the nodes of the trees correspond to maximal cliques of $G$. This reasoning motivates the concept a \textit{clique tree}, a tree that has maximal cliques in $G$ as its nodes and where for any $v \in V(G)$, the subgraph of the tree induced by the nodes containing $v$ is connected. More precisely, for a chordal graph $G$ we let $\mathcal{M}_G$ denote the set maximal cliques of $G$, and we say that the pair $(T, \sigma)$ is a \textit{clique tree pair} of graph $G$ if $T$ is a tree and $\sigma: V(T) \rightarrow \mathcal{M}_G$ is a bijection such that for any $v \in V(G)$, the subgraph of $T$ induced by the nodes in $\{ u \in V(T): v \in \sigma(u) \}$ is a subtree of $T$. Clearly, a graph $G$  has tree $T$ as a clique tree if any only if it has $(T, \sigma)$ as a clique tree pair for some bijection $\sigma: V(T) \rightarrow \mathcal{M}_G$. 

We can now state Gavril's characterization of chordal graphs in terms of intersection graphs and clique trees.

\begin{theorem}[Gavril \cite{Gavril1974}]\label{gavril} \label{intersection} For a graph $G$, the following are equivalent:
\begin{itemize}
    \item $G$ is chordal.
\item $G$ has a clique tree.\label{p3}
    \item $G$ is an intersection graph of a family of subtrees of a tree. 
\end{itemize}
\end{theorem}

Since the threshold graphs form a subclass of chordal graphs, it is natural to ask what the possible clique trees of threshold graphs are. For the class of split graphs, which is a superclass of threshold graphs, McMorris and Shier \cite{McMorris1983} showed that a chordal graph is a split graph if and only if it has a star graph (the bipartite graph $K_{1,n}$ for some positive integer $n$) as a clique tree. Therefore, any threshold graph has a star graph as a clique tree.

In order to state the characterization of the clique trees of threshold graphs, we need one more piece of notation. 
For a threshold graph $G$ with set $\mathcal{M}_G$ of maximal cliques, for positive integer $j$ we let $\mathcal{K}_G^j = \{ M \in \mathcal{M}_G: |M| \geq j\}$ be the set of maximal cliques in $G$ of size at least $j$.  
For instance, in the threshold graph $G$ in Figure \ref{fig: threshold ex2}, $\mathcal{M}_G$ consists of five cliques: the clique consisting of the left twelve vertices, and $N_G[v_i]$ for $i \in [4]$. We have $|\mathcal{K}_{G}^j|=0$ for $j>12$, $|\mathcal{K}_{G}^j|=1$ for $j \in \{ 7,8,\ldots, 12\}$, $|\mathcal{K}_{G}^6|=2$,  $|\mathcal{K}_G^{5}|=3$, and $|\mathcal{K}_G^{j}|=5$ for $j \leq 4$. 
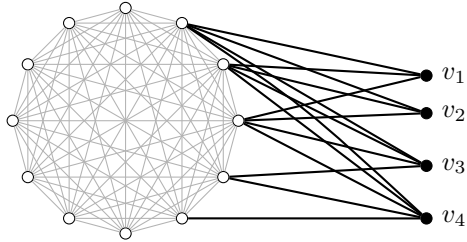
\begin{figure}[H]
\begin{center}
\begin{tikzpicture}

    \foreach \i in {1,...,12} {
        \node[draw=black, fill=white, circle, inner sep=1.5pt] (b\i) at ({2 + 1.5*cos(30*\i)}, {3 + 1.5*sin(30*\i)}) {};
    }

    \foreach \i in {1,...,12} {
        \foreach \j in {1,...,12} {
            \ifnum\i<\j
                \draw[black!30] (b\i) -- (b\j);
            \fi
        }
    }

    \def\xshift{6.0}

    \node[draw=black, fill=black, circle, inner sep=1.5pt, label=right:{$v_{2}$}] (bv1) at (\xshift,3.1) {};
    \node[draw=black, fill=black, circle, inner sep=1.5pt, label=right:{$v_{1}$}] (bv2) at (\xshift,3.6) {};
    \foreach \i in {12,1,2} {
        \draw[black, thick] (bv1) -- (b\i);
        \draw[black, thick] (bv2) -- (b\i);
    }

    \node[draw=black, fill=black, circle, inner sep=1.5pt,label=right:{$v_3$}] (gv) at (\xshift,2.4) {};
    \foreach \i in {11,12,1,2} {
        \draw[black, thick] (gv) -- (b\i);
    }

    \node[draw=black, fill=black, circle, inner sep=1.5pt,label=right:{$v_4$}] (rv) at (\xshift,1.7) {};
    \foreach \i in {10,11,12,1} {
        \draw[black, thick] (rv) -- (b\i);
    }
    \draw[black, thick] (rv) .. controls ($(rv)!0.5!(b2) + (0,0.2)$) .. (b2);

\end{tikzpicture}
\end{center}
\caption{A threshold graph with five maximal cliques.}
\label{fig: threshold ex2}
\end{figure}

We can now characterize the clique tree pairs of threshold graphs. Unsurprisingly, the characterization is based on a nested subtree condition, analogous to the nested neighborhood condition characterizing threshold graphs.

\begin{theorem}\label{thm: star}
    Let $G$ be a threshold graph with clique number $\omega(G)$, let $T$ be a tree, and let $\sigma: V(T) \rightarrow \mathcal{M}_G$ be a bijection. 
    For positive integer $j \in [\omega(G)]$, let $T_j$ denote the subgraph of $T$ induced by the nodes in $\{ \sigma^{-1}(M): M \in \mathcal{K}_G^j\}$.
    
    Then $(T, \sigma)$ is a clique tree pair of $G$ if and only if for each $j \in \{ 1, \ldots, \omega(G)\}$, $T_j$ is a subtree of $T$ and moreover for $j \in [\omega(G)-1]$, $T_{j+1}$ is a subtree of $T_{j}$.
\end{theorem}
\begin{proof}
If $G$ is isomorphic to a complete graph, then $|\mathcal{M}_G|=1$ and so the theorem statement trivially holds.
    Now assume that $G$ is a threshold graph that is not isomorphic to a complete graph. Let $T$ be a tree and let $\sigma: V(T) \rightarrow \mathcal{M}_G$ be a bijection.
    Letting $K$ denote a clique in $G$ of size $\omega(G)$ and with $I=V(G) \setminus K$, which is nonempty since $G$ is not isomorphic to a complete graph, we have that $G$ is $(K,I)$-split and the vertices of $I$ can be listed as $v_1, \ldots, v_{|I|}$ so that for each $i \in \{ 1,2, \ldots, |I|-1\}$ we have $N_G(v_i) \subseteq N_G(v_{i+1}) \subsetneq K$.
    Hence, $\mathcal{M}_G$ consists of the clique $K$ and the closed neighborhoods $N_G[v_i]$ for each $i \in [|I|]$. 


    Let $m$ denote the number of distinct elements of the collection $\{ N_G(v_1), \ldots, N_G(v_{|I|}) \}$, so $m \leq |I|$. Let $(N_1, \ldots, N_{m})$ be the $m$-tuple consisting of the $m$ distinct elements of $\{ N_G(v_1), \ldots, N_G(v_{|I|}) \}$, arranged in increasing order of size. Thus, we have  $N_1 \subsetneq  \ldots \subsetneq N_m \subsetneq K$.
    
    Let $z_1$ be any vertex in $N_1$ and for each $i \in \{2,\ldots, m\}$, let $z_i \in K$ be any vertex in $N_i \setminus N_{i-1}$. Then for any $i \in [m]$, we have $|N_i|<\omega(G)$ and a maximal clique $M \in \mathcal{M}_G$ contains $z_i$ if and only if $M \in \mathcal{K}_G^{|N_i|+1}$.

    First suppose that $(T, \sigma)$ is a clique tree pair of $G$. Then each $T_{z_i}$ is a subtree of $T$, so for each $j \in [\omega(G)-1]$, we have that $T_j$ is a subtree of $T$. Moreover $T_{\omega(G)}$ is a subtree of $T$ because every maximum clique in $G$ intersects $K$. Hence, for any $j \in [\omega(G)]$, $T_j$ is a subtree of $T$. 
    Moreover, for $j \in [\omega(G)-1]$ we have $\mathcal{K}_G^j = \mathcal{K}^{j+1}_G \cup \{ M \in \mathcal{M}_G: |M|=j\}$, so $V(\mathcal{K}^{j+1}_G) \subseteq V(\mathcal{K}^j_G)$. Since $T_{j+1}$ and $T_j$ are both subtrees of $T$, it follows that $T_{j+1}$ is a subtree of $T_j$.

    Conversely, suppose that $T$ is a tree, $\sigma:V(T)\rightarrow \mathcal{M}_G$ is a bijection, $T_j$ is a subtree of $T$ for any $j \in [\omega(G)]$, and for $j \in [\omega(G)-1]$ we have that $T_{j+1}$ is a subtree of $T_{j}$. Fix $v \in V(G)$ and let $H_v$ denote the subgraph of $T$ induced by the nodes in $\{ \sigma^{-1}(M):v \in M \}$. 
    We will show that $H_v$ is a subtree of $T$. If $v \notin K$, there is exactly one maximal clique in $G$ containing $v$, namely $N_G[v]$. Similarly, if $v\in K \setminus \bigcup_{i=1}^{|I|}N_G(v_i)$, the vertex $v$ is in only one maximal clique in $G$, namely $K$. In these two cases, $|V(H_v)|=1$ so $H_v$ is a subtree of $T$. Finally, if $v \in \bigcup_{i=1}^{|I|}N_G(v_i)$, let $\ell \in [m]$ be the smallest index where $v \in N_{\ell}$, so for any $M \in \mathcal{M}_G$, we have $v \in M$ if and only if $M \in \mathcal{K}_G^{|N_{\ell}|+1}$. In other words, we have $H_v=T_{j}$ where $j=|N_{\ell}|+1$. By the assumption that $T_j$ is a subtree of $T$, it follows that $H_v$ is a subtree of $T$. Hence, $(T, \sigma)$ is a clique tree pair of $G$.
\end{proof}

As a corollary of Theorem \ref{thm: star}, we can prove that threshold graphs are ``universal" in the sense that any threshold graph $G$ has any tree on $|\mathcal{M}_G|$ nodes as a clique tree. We say that a chordal graph $G$ is \textit{clique tree universal} if any tree on $|\mathcal{M}_G|$ nodes is a clique tree for $G$.

\begin{corollary} \label{cor: threshold clique tree} Every threshold graph is clique tree universal.
\end{corollary}
\begin{proof}
Suppose $G$ is a threshold graph whose vertices can be partitioned into a clique $K$ and independent set $I$ so that $I$ has the nested neighborhood property in $G$. 
    Given a tree $T$ on $|\mathcal{M}_G|$ nodes, let $r$ be an arbitrary node of $T$. We will construct a bijection $\sigma: V(T) \rightarrow \mathcal{M}_G$. First, set $\sigma(r)=K$. Then, perform a breadth-first search from $r$, and upon encountering a new node $u$ of $T$, set $\sigma(u)=M$ where $M \in \mathcal{M}_G$ is any element of $\mathcal{M}_G$ of maximum size among the elements of $\mathcal{M}_G$ that have not already been assigned to a node of $T$ under $\sigma$. It is straightforward to verify that this process produces a bijection $\sigma: V(T) \rightarrow \mathcal{M}_G$ satisfying the conditions of Theorem \ref{thm: star}, so $T$ is a clique tree of $G$.
\end{proof}

Corollary \ref{cor: threshold clique tree} extends the corollary of McMorris and Shier's result that any threshold graph has a star graph as a clique tree. The clique tree universality property of threshold graphs is interesting because it does not hold for general chordal graphs, or even for general split graphs. For instance, the split graph $G$ in Figure \ref{fig: split ex} has $|\mathcal{M}_G|=4$ but does not have the path on four vertices as a clique tree.

\begin{figure}[H]
\begin{center}
\begin{tikzpicture}

    \node[draw=black, fill=white, circle, inner sep=1.5pt] (v1) at (0,1.2) {};
    \node[draw=black, fill=white, circle, inner sep=1.5pt] (v2) at (-0.6,0) {};
    \node[draw=black, fill=white, circle, inner sep=1.5pt] (v3) at (0.6,0) {};

    \draw[black] (v1) -- (v2);
    \draw[black] (v1) -- (v3);
    \draw[black] (v2) -- (v3);

    \node[draw=black, fill=black, circle, inner sep=1.5pt] (u1) at (3,1.2) {};
    \node[draw=black, fill=black, circle, inner sep=1.5pt] (u2) at (3,0.6) {};
    \node[draw=black, fill=black, circle, inner sep=1.5pt] (u3) at (3,0.0) {};

    \draw[black, thick] (u1) -- (v1);
    \draw[black, thick] (u2) -- (v1);
    \draw[black, thick] (u2) -- (v3);
    \draw[black, thick] (u3) -- (v3);

    \draw[black, thick] 
        (u1) .. controls (-0.6,1.8)  .. (v2);

    \draw[black, thick] 
        (u3) .. controls (-0.6,-0.6) .. (v2);

\end{tikzpicture}
\end{center}
\caption{A split graph that is not a threshold graph and is not clique tree universal.}
\label{fig: split ex}
\end{figure}
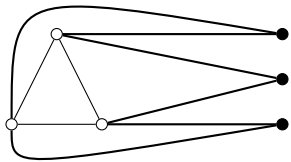

It is natural to ask whether threshold graphs are the only clique tree universal chordal graphs, but this is not the case. For example, if one removes an edge incident to a vertex of degree two in the graph from Figure \ref{fig: split ex}, the resulting graph is not threshold but is clique tree universal. An interesting direction for future work would be to characterize the clique tree universal chordal graphs in general.

\section*{Acknowledgments}
The author thanks the referees for valuable comments and for suggesting to apply Theorem \ref{thm: conditions} to chordal graphs. The author is also grateful to Xingxing Yu for helpful feedback throughout the research process.

\end{document}